\documentclass{imanum}

\usepackage{float}

\usepackage{amsthm}
\usepackage{amsmath}
\usepackage{amssymb}
\usepackage{mathpazo}
\usepackage{dsfont}
\usepackage[utf8]{inputenc}
\usepackage{epstopdf}

\usepackage{ulem}

\usepackage{alphalph}
\def\theequation{\AlphAlph{\value{\equation}}}

\usepackage{algorithm}
\usepackage{algorithmic}
\usepackage{enumerate}
\usepackage{mathtools}
\usepackage{xcolor} 
\definecolor{cadmiumgreen}{rgb}{0.0, 0.42, 0.24}

\newcommand{\JN}[1]{\textcolor{black}{#1}}

\newcommand{\LB}[1]{\textcolor{black}{#1}}

\def\Re{\mathrm{Re}\,}

\newcommand{\R}{\mathbb{R}}

\newcommand{\lb}{\Delta_{\Gamma}}
\newcommand{\pnx}{\partial_{n_x}}
\newcommand{\pt}{\partial_t}

\newcommand{\pn}{\partial_n}

\newcommand{\utot}{u^{\text{tot}}}
\newcommand{\uscat}{u}
\newcommand{\uscatau}{u^{\tau}}
\newcommand{\uscatf}{u^{\text{scat}}}
\newcommand{\uinc}{u^{\text{inc}}}
\newcommand{\ginc}{g^{\text{inc}}}

\newcommand{\Bimp}{B_\text{imp}}

\newcommand{\wginc}{\widehat{g}^{\text{inc}}}
\newcommand{\wuinc}{\widehat{u}^{\text{inc}}}

\newcommand{\wvarphi}{\widehat{\varphi}}
\newcommand{\wpsi}{\widehat{\psi}}

\newcommand{\wu}{\widehat{u}}
\newcommand{\wuscat}{\widehat{u}}

\newcommand{\mH}{\mathcal{H}}
\newcommand{\mX}{\mathcal{X}}

\newcommand{\eps}{\varepsilon}

\newcommand{\norm}[1]{\left\lVert#1\right\rVert}

\newcommand{\abs}[1]{\left|#1\right|}
\newcommand{\snorm}[1]{\left|#1\right|}

\def\bcl{\color{black}}
\def\ecl{\color{black}}
\jno{ }
\received{ }
\revised{ }

\numberwithin{equation}{section}

\begin{document}
\title{Time-dependent acoustic scattering from generalized impedance boundary conditions via boundary elements and convolution quadrature}
\shorttitle{Scattering from generalized impedance boundary conditions}
\author{%
{\sc
Lehel Banjai \thanks{ Email: L.Banjai@hw.ac.uk},
} \\[2pt]
The Maxwell Institute for Mathematics in the Sciences; School of Mathematical \& Computer \\
Sciences, Heriot-Watt University, EH14 4AS Edinburgh, UK\\[6pt]
{\sc Christian Lubich}\thanks{Email: lubich@na.uni-tuebingen.de}\\[2pt]
Mathematisches Institut, Universität Tübingen, Auf der Morgenstelle, D-72076 Tübingen,\\
Germany \\[6pt]
{\sc and}\\[6pt]
{\sc J\"org Nick}\thanks{Email: nick@na.uni-tuebingen.de}\\[2pt]
Mathematisches Institut, Universit\"at T\"ubingen, Auf der Morgenstelle, D-72076 T\"ubingen,\\
Germany
}
\shortauthorlist{L. Banjai, C. Lubich and J. Nick}


\maketitle
\ \\
\begin{abstract}
\noindent{Generalized impedance boundary conditions, wave equation, exterior domain, boundary elements, convolution quadrature.}
\\ \noindent 
Generalized impedance boundary conditions are effective, approximate boundary conditions that describe scattering of waves in situations where the wave interaction with the material involves multiple scales. In particular, this includes materials with a thin coating (with the thickness of the coating as the small scale) and strongly absorbing materials. For the acoustic scattering from generalized impedance boundary conditions, the approach taken here first determines the Dirichlet and Neumann boundary data from a system of time-dependent boundary integral equations with the usual boundary integral operators, and then the scattered wave is obtained from the Kirchhoff representation. The system of time-dependent boundary integral equations is discretized by boundary elements in space and convolution quadrature in time. The well-posedness of the problem and the stability of the numerical discretization rely on the coercivity of the Calder\'on operator for the Helmholtz equation with frequencies in a complex half-plane. Convergence of optimal order in the natural norms is proved for the full discretization. Numerical experiments illustrate the behaviour of the proposed numerical method.

%
%
\end{abstract}
\section{Introduction}
In this paper we study a numerical approach to compute time-dependent acoustic scattering from obstacles that, due to their particular material properties, yield effective boundary conditions known as {\it generalized impedance boundary conditions}.

 On an exterior domain $\Omega\subset\R^3$, the complement of one or several bounded domains, we consider the acoustic wave equation (with wave speed $c=1$ in appropriate physical units)
\begin{equation}\label{wave-eq}
\pt^2 \utot - \Delta \utot =0 \quad \text{ in the exterior domain }\, \Omega.
\end{equation}
Given an incident wave $\uinc$, which is a solution of the wave equation in $\R^3$ with support in $\Omega$ at time $t=0$, the objective is to compute the scattered wave $\uscatf(\cdot,t)\in H^1(\Omega)$ on a time interval $0\le t \le T$ (possibly only at a few selected space points $x\in\Omega$), such that the total wave $\utot=\uscatf+\uinc$ is a solution to the wave equation \eqref{wave-eq} that satisfies the specified boundary conditions on the boundary $\Gamma=\partial\Omega$. \LB{As we will be devising a numerical method for the computation of the scattered field $\uscatf$, we will from now on denote it simply by $u = \uscatf$.}

The first boundary condition we consider is a model for a material with a {\it thin coating} of width $\eps$.  The time-dependent version of an effective boundary condition, given by \cite{EN93} for the time-harmonic case, reads (with $\pn u$ denoting the outer normal derivative)
\begin{align*}
\text{(A)} \qquad\qquad\qquad \pn \utot&=\eps(\pt^2\utot -\lb \utot) \quad\text{ on }\Gamma.\qquad\qquad
\end{align*}
An effective boundary condition for a {\it strongly absorbing material}, as studied by  \cite{NN15}, is given by
\begin{equation*}
\text{(B1)} \qquad\qquad\qquad\quad\  \pn \utot =\frac1\varepsilon\pt^{1/2}\utot \quad\text{ on }\Gamma.
\qquad\qquad\qquad
\end{equation*}
\LB{This is an order 1 approximation of the Dirichlet-to-Neumann operator for the equation
\[
\partial_t^2 \utot +\frac1{\varepsilon^2} \partial_t \utot -\Delta \utot  = 0
\qquad \text{in the interior domain } \mathbb{R}^3 \setminus \overline\Omega.
\]
An order 2 boundary condition derived in \cite{HJN05} in the frequency domain and translated to the time-domain setting reads
\begin{equation*}
\text{(B2)} \qquad\qquad\quad\  \pn \utot =\frac1\varepsilon\pt^{1/2}\utot-\mathcal{H}\utot \quad\text{ on }\Gamma,
\qquad\qquad\qquad
\end{equation*}
where $\mathcal{H}$ is the mean curvature of $\Gamma$ and $\Gamma$ is assumed to be smooth.}
Another example of interest are {\it acoustic boundary conditions}; see \cite{B74}. 
They are formulated as a coupled system, which reads 
\begin{align*}
\begin{split}
\text{(C)} \qquad\quad\ 
m \pt^2 \delta + \alpha \pt\delta +k \delta+ \pt \utot &=0\\
\pt \delta &=-\pn \utot
\end{split}
 \qquad\text{ on }\Gamma,
\end{align*}
taken with zero initial conditions for $\delta$ and $\pt\delta$.
Here, $m>0$, $\alpha\ge 0$ and $k>0$ are given mass, damping and stiffness parameters, respectively.

To our knowledge, no numerical analysis is so far available for acoustic scattering from any of the boundary conditions (A), (B) or (C) or other generalized impedance boundary conditions as studied, e.g., in \cite{AA96}, or for electromagnetic scattering from generalized impedance boundary conditions as studied, e.g., in \cite{AH97,DHJ06}. Numerical approaches to time-dependent acoustic scattering have been studied for Dirichlet and Neumann boundary conditions, using either {\it space-time Galerkin methods} based on the variational formulation of \cite{BH86a,BH86b} or {\it boundary element space discretization coupled with convolution quadrature time discretization} based on the theoretical framework of \cite{L94,LS09}. Here we shall adopt the latter approach and study it for the numerical treatment of acoustic scattering from generalized impedance boundary conditions such as (A)--(C).

In Section~\ref{sec:analytic} we show that the boundary conditions (A)--(C) fall into a general class of convolutional boundary conditions with a transfer operator of positive type. We present this general framework and derive the well-posedness of the scattering problem for such boundary conditions, using a system of time-dependent boundary integral equations suitable for numerical discretization (similar to  \cite{BR18}), the Kirchhoff representation formula and relying on Laplace transform techniques.

In Section~\ref{sec:cq} we study the semidiscretization in time of the time-dependent boundary integral equations and the Kirchhoff representation formula by convolution quadrature. This is shown to yield a stable and convergent approach. 

In Section~\ref{sec:bem} we study the semidiscretization in space of the time-dependent boundary integral equations by boundary elements. This is also shown to yield a stable and convergent approach. 

In Section~\ref{sec:cq-bem} we study the full discretization by boundary elements in space and convolution quadrature in time and combine the results of Sections \ref{sec:cq} and~\ref{sec:bem} to prove our main result, which shows that the method is stable and convergent of optimal order in the natural norms. This result becomes possible by the coercivity property of the Calder\'on operator for the Helmholtz equation as shown in \cite{BLS15} together with the positivity property of the transfer operator of the boundary condition, by the discrete operational calculus of convolution quadrature, and by using known approximation error bounds for convolution quadrature and boundary elements.

In Section~\ref{sec:num} we present numerical examples for scattering from the generalized impedance boundary conditions (A)--(C).
\renewcommand{\theequation}{\arabic{section}.\arabic{equation}}
\section{ Framework and analytical background}
\label{sec:analytic}
In the following, we introduce basic notation and collect analytical tools that will allow us to formulate the discretization and give its error analysis in the later chapters. The generalized impedance boundary conditions will be written as
\begin{align}\label{gibc}
\pn \utot = F(\pt)\pt \utot,
\end{align}
where $F(\pt)$ is a temporal convolution operator of positive type. \LB{Hence, recalling that $\uinc(0)$ is supported inside $\Omega$, the scattered field $\uscat = \utot-\uinc$ solves the wave equation with zero initial condition and a generalized impedance boundary condition
  \begin{equation}
     \label{eq:scat_wave}   
  \begin{aligned}
    \partial_t^2 \uscat -\Delta \uscat &= 0 & &\text{in } \Omega \times [0,T]    \\
    \uscat(0) = \partial_t \uscat(0) &= 0 & &\text{in } \Omega\\
      \pn \uscat-F(\pt)\pt\uscat &= \ginc & & \text{on } \Gamma \times [0,T],
\end{aligned}
  \end{equation}
where $\ginc = F(\pt)\pt\uinc-\pn \uinc$.
}

In Subsection~\ref{subsec:gibc} we give the precise formulation in an abstract setting and show that examples (A), (B) and (C) fit into the presented framework. In Subsection~\ref{subsec:boundary operators} we recall the potential and boundary integral operators of the Helmholtz equation. In Subsection~\ref{subsec:time-harmonic} we show the well-posedness of the {\it time-harmonic} scattering problem from generalized impedance boundary conditions, with estimates that trace carefully the dependence on complex frequencies varying in a half-plane. This relies essentially on a coercivity result for the Calder\'on operator derived in \cite{BLS15}, in combination with the properties of the generalized impedance boundary condition. Using Laplace transform techniques, we then obtain a well-posedness result for the {\it time-dependent} scattering problem in Subsection~\ref{subsec:time-dep}.

\subsection{Generalized impedance boundary conditions}
\label{subsec:gibc}

(i) {\it Hilbert space setting.} Let $X\subset H^{1/2}\left( \Gamma \right)$ be a complex Hilbert space, densely and continuously embedded in $H^{1/2}\left( \Gamma \right)$, equipped with a seminorm $\abs{w}^2_X$ and the norm 
\begin{equation}\label{X-norm}
\norm{w}^2_X=\norm{w}^2_{H^{1/2}(\Gamma)}+\abs{w}^2_X.
\end{equation}
Let $X'$ be its dual, with the anti-dual pairing $\langle \cdot,  \cdot \rangle$ of $X$ and $X'$, which is chosen to coincide with the $L^2(\Gamma)$ inner product (anti-linear in the first argument) on $X\times L^2(\Gamma)$, which is possible due to the chain of dense and continuous inclusions
\begin{align*}
X\subset H^{1/2}(\Gamma)\subset
L^2(\Gamma)= L^2(\Gamma)' \subset H^{-1/2}(\Gamma)
\subset X'.
\end{align*}

(ii) {\it Transfer operators and temporal convolution.} Let $F(s)\colon X \rightarrow X'$, for all complex $s$ with positive real part, be an analytic family of bounded linear operators.
We assume that $F$ is {\it polynomially bounded:}\/ There exists a real $\mu$ and, for every $\sigma>  0$, there exists $M_\sigma<\infty$ such that
\begin{align}\label{F-bound}
\norm{F(s)}_{X'\leftarrow X}&\leq M_{\sigma} \abs{s}^\mu \quad  \text{ for \ Re } s \ge \sigma.
\end{align}
This condition  ensures that $F$ is the Laplace transform of a distribution of finite order of differentiation with support on the nonnegative real half-line $t \ge 0$. For a function $g:[0,T]\to X$, which together with its extension by $0$ to the negative real half-axis is sufficiently regular, we use the Heaviside operational calculus notation
\begin{equation} \label{Heaviside}
F(\pt)g = (\mathcal{L}^{-1}F) * g
\end{equation}
to denote the temporal convolution of the inverse Laplace transform of $F$ with $g$. The motivation for this very useful notation comes from the facts that for $F(s)=s$, i.e. $\mathcal{L}^{-1}F$ is the derivative of Dirac's delta distribution, we have $F(\pt)g=\pt g$, and for two such families of operators $F(s)$ and $G(s)$ mapping into compatible spaces, the associativity of convolution yields the composition rule $F(\pt)G(\pt)q = (FG)(\pt)q$.

For real $r$, we let $H^r(\R,X)$ be the Sobolev space of order $r$ of $X$-valued functions on $\R$, and on finite intervals $(0, T )$ we denote
$$
H_0^r(0,T;X) = \{g|_{(0,T)} \,:\, g \in H^r(\R,X)\ \text{ with }\ g = 0 \ \text{ on }\ (-\infty,0)\} . 
$$
Equivalent to the natural norm on $H_0^r(0,T;X)$ is the norm $\| \pt^r g \|_{L^2(0,T;X)}$.
The Plancherel formula yields the following \cite[Lemma 2.1]{L94}:
If $F(s)$ is bounded by \eqref{F-bound} in a half-plane $\text{Re }s > 0$, then $F(\pt)$ extends by density to a bounded linear operator
\begin{equation}\label{sobolev-bound}
F(\pt) : H^{r+\mu}_0(0,T;X) \to H^r_0(0,T;X')
\end{equation}
for arbitrary real $r$. Note that for $r>1/2$ we have $H^r_0(0,T;X')\subset C([0,T];X')$.

(iii) {\it Positivity condition.}\/ As a key condition, we assume that for every $\sigma> \LB{\sigma_0 \geq} 0$, there exists $c_\sigma>0$ such that
\begin{equation}\label{F-pos}
\text{Re}\,\langle \psi, F(s)\psi \rangle \ge c_\sigma \, \bigl| s^{-1}\psi \bigr| _X^2
\quad \text{for all } \psi \in X \text{  and }\text{ Re }s\ge \sigma.
\end{equation}

\begin{lemma} \label{lem:gibc}
\noindent The boundary conditions (A)--(C) are special cases of generalized impedance boundary conditions \eqref{gibc} with transfer operators $F(s):X\to X'$, $\text{Re}\, s >0$, that satisfy a polynomial bound \eqref{F-bound} and the positivity condition \eqref{F-pos}.
\end{lemma}

\begin{remark} \label{rem:gibc} For (A) we choose $X=H^1(\Gamma)$ equipped with the seminorm $\abs{\psi}_X=\eps^{1/2}\norm{\nabla_\Gamma \psi}_{L^2(\Gamma)^3}$ and the corresponding norm \eqref{X-norm}. This yields \eqref{F-bound} and \eqref{F-pos} with constants that are independent of $\eps\in(0,1]$.
For (B) and (C) we have simply $X=H^{1/2}(\Gamma)$ and the seminorm $|\cdot|_X$ is identically zero. \LB{In \eqref{F-pos}, we set $\sigma_0 =  \min(0,2\eps\mathcal{H}_{\mathrm{max}})^2$ with $\mathcal{H}_{\mathrm{max}}=\max_\Gamma \mathcal{H}$ for (B2). For (A), (B1) and (C), we can set $\sigma_0 = 0$.  In all cases, the constant $M_\sigma$, and for (A) the constant $c_{\sigma}^{-1}$, grow at most polynomially as $\sigma \rightarrow 0$.}
\end{remark}

\begin{proof}
(A) \  We begin by rearranging the boundary condition in question to obtain
\begin{align*}
\pn u&=\eps(\pt^2u -\lb u)
\\&=\eps(\pt-\lb \pt^{-1})\pt u.
\end{align*}
The corresponding transfer operator  is given by 
\begin{align*}
F(s)=\eps\bigl(s-\lb s^{-1}\bigr).
\end{align*}
For this operator, we work with the space
$X=H^1(\Gamma)$,
where we choose the semi-norm $\abs{\cdot}_X$ as
\begin{align*}
\abs{\psi}_X^2=\eps\norm{\nabla_\Gamma \psi}_{L^2(\Gamma)^3}^2
\end{align*}
and the norm on $X$ as \eqref{X-norm}.
Then, $F$ satisfies the polynomial bound \eqref{F-bound} with $\mu=1$ and a constant that is independent of~$\eps$ for $0<\eps\le 1$
, since
$$
\| F(s) \|_{X' \leftarrow X} =  \sup_{\psi_1,\psi_2\in X \atop \| \psi_1 \|_X = \| \psi_2 \|_X =1} \abs{\langle \psi_2, F(s) \psi_1 \rangle}
$$
and, for $\text{Re}\, s \ge \sigma>0$,
\begin{align*}
 \abs{\langle \psi_2, \eps\bigl(s-\lb s^{-1}\bigr) \psi_1 \rangle }
 &\le \abs{ \eps \, |s|\, ( \psi_2, \psi_1)_{L^2(\Gamma)} + \eps \, |s|^{-1}\, ( \nabla_\Gamma\psi_2, \nabla_\Gamma\psi_1)_{L^2(\Gamma)^3} }
 \\
 & \le 
\eps \, |s|\, \| \psi_2\|_{L^2(\Gamma)} \, \| \psi_1\|_{L^2(\Gamma)} +
 \eps \, |s|^{-1}\, \| \nabla_\Gamma \psi_2\|_{L^2(\Gamma)^3} \, \| \nabla_\Gamma \psi_1\|_{L^2(\Gamma)^3}
 \\
 & \le |s|\, (\eps + \sigma^{-2}) \, \| \psi_2\|_{X} \, \| \psi_1\|_{X}.
\end{align*}
$F$ satisfies the positivity condition \eqref{F-pos}, because for $\psi\in X$ and $\text{Re}\, s \ge \sigma>0$,
\begin{align*}
\text{Re}\, \langle \psi,  \eps\bigl(s-\lb s^{-1}\bigr) \psi \rangle 
& =  \eps \,\text{Re}\,s\, \| \psi \|_{L^2(\Gamma)}^2 + \eps  \, \text{Re}\,\overline s\,  \| s^{-1} \nabla_\Gamma \psi \|_{L^2(\Gamma)^3}^2 
\\
& \ge \sigma \eps \, \| s^{-1} \nabla_\Gamma \psi \|_{L^2(\Gamma)^3}^2 
\\
&= \sigma \, \bigl| s^{-1} \psi \bigr|_X^2.
\end{align*}

\noindent (B) \ \ 
\LB{Here the transfer operator is either $F(s) = \varepsilon^{-1}s^{-1/2}$ or $F(s) = \varepsilon^{-1}s^{-1/2}-\mathcal{H}s^{-1}$}.  \JN{We choose $X=L^2(\Gamma)$ and the corresponding weighted norm as the seminorm 
\begin{align*}
\abs{\psi}_X = \varepsilon^{-1/2}\norm{\psi}_{L^2(\Gamma)}.
\end{align*}
Like in (A) this yields, for arbitrary $\psi_1,\psi_2$ with $ \| \psi_1 \|_X = \| \psi_2 \|_X =1$, the following estimate
\begin{align*}
 \abs{\left \langle \psi_2, \left(\varepsilon^{-1}s^{-1/2}+\mathcal{H}s^{-1}\right) \psi_1 \right\rangle }
 &\le
 \left(\varepsilon^{-1}\abs{s}^{-1/2}+\norm{\mathcal{H}}_{L^\infty(\Gamma)}\abs{s}^{-1}\right)
 \norm{\psi_2}_{L^2(\Gamma)}\norm{ \psi_1 }_{L^2(\Gamma)} 
 \\&\le 
 \left(\abs{s}^{-1/2}+\varepsilon\norm{\mathcal{H}}_{L^\infty(\Gamma)}\abs{s}^{-1}\right)
 \le \sigma^{-1}M,
\end{align*}
where the last estimate holds true for  $\sigma<1$. 
Therefore, both (B1) and (B2) satisfy
 \eqref{F-bound} with $\mu = 0$, $\sigma_0=0$ and constants $M_\sigma$ independent of $\varepsilon$. 
 }
 \bcl For (B1), \eqref{F-pos} holds with $\sigma_0 = 0$ since  $\text{Re}\, s^{-1/2}>0$ for $\text{Re}\, s >0$. For (B2) we have, for  $\Re s>  0$,
 \JN{
\[
  \begin{split}    
\Re \langle \psi, F(s) \psi \rangle  
&
\geq \left(\varepsilon^{-1} \Re s^{-1/2}  - \mathcal{H}_{\mathrm{max}} \,\Re s^{-1}\right)
\|\psi\|^2_{L^2(\Gamma)}\\
&= |s|\left(\varepsilon^{-1} \Re s^{1/2}- \mathcal{H}_{\mathrm{max}} \,\frac{\Re s}{|s|}\right)\|s^{-1}\psi\|_{L^2(\Gamma)}^2
\geq 
\dfrac{\sigma^{3/2}}{2}
\abs{s^{-1}\psi}_{X}^2
  \end{split}
\]
}
\JN{for $\Re s^{1/2} \ge2\eps \mathcal{H}_{\mathrm{max}}$,  which holds true for
 $\Re s \geq \sigma_0 =  \max\left(0,4\varepsilon^2\mathcal{H}_{\mathrm{max}}^2\right)$.}\ecl

\ \\
(C) \ \ For the acoustic boundary condition, substituting the second equation into the first equation yields the equivalent formulation
\begin{align*}
(m \pt  + \alpha+k\pt^{-1}) \pn u- \pt u &=0,
\end{align*}
and by applying the inverse we obtain the formulation  \eqref{gibc} with
\begin{align*}
 F(s) = (m s  + \alpha+ks^{-1})^{-1}.
\end{align*}
We again choose $X=H^{1/2}(\Gamma)$ with the seminorm 
$\abs{\cdot}_X\equiv 0$.
We clearly have the bound \eqref{F-bound} with $\mu=-1$ and we also have the positivity property \eqref{F-pos},  because
for $\text{Re}\, s >0$,
\begin{align*}
\text{Re } F(s)=
\text{Re } (m s  + \alpha+ks^{-1})^{-1}
&=\text{Re }\dfrac{(m\overline{s}+\alpha+k \overline{s}^{-1})}{\abs{m\overline{s}+\alpha+k \overline{s}^{-1}}^2}
\ge 0.
\end{align*}
This non-negativity yields \eqref{F-pos} for $\abs{\cdot}_X\equiv 0$.
\end{proof}

\subsection{Recap: Kirchhoff representation formula and Calder\'on operator for the Helmholtz equation}
\label{subsec:boundary operators}
With  the Laplace transformed wave equation, i.e. the Helmholtz equation
\begin{equation}\label{helmholtz}
s^2 \widehat{u}-\Delta \widehat{u} =0,\qquad \Re s>0,
\end{equation}
we associate the usual boundary integral operators in the notation used, e.g., by \cite{LS09,S16,BLS15}.
The single layer potential operator is denoted by
\begin{align*}
	S(s)\varphi(x)&=\int_{\Gamma}\dfrac{1}{4\pi\left| x-y \right|}e^{-s\left|x-y\right|}\varphi(y) d\Gamma_y \quad \quad x\in  \Omega,
\end{align*}
and the double layer potential is denoted by
\begin{align*}
	D(s)\psi(x)&=\int_{\Gamma}\left(\partial_{n_y}\dfrac{1}{4\pi\left| x-y \right|}e^{-s\left|x-y\right|}\right)\psi(y) d\Gamma_y \quad \quad x\in  \Omega.
\end{align*}
These integral operators are bounded linear operators on the spaces
$$
S(s): H^{-1/2}(\Gamma) \to H^1(\Omega), \quad\
D(s): H^{1/2}(\Gamma) \to H^1(\Omega).
$$
They are bounded, for $\text{Re} \ s \ge \sigma >0$, by $C_\sigma\abs{s}^{\mu}$ with $\mu=1$ and $\mu=3/2$, respectively; see \cite{BH86a,BH86b,LS09}.

Every solution $\wu\in H^1(\Omega)$ of the Helmholtz equation \eqref{helmholtz} can be written in terms of the single and double layer potentials of the boundary values by the representation formula
\begin{equation}\label{kirchhoff}
\wu=S(s)\wvarphi+D(s)s^{-1}\wpsi,
\end{equation}
where $\wvarphi$ and $\wpsi$ are the (scaled) Neumann and Dirichlet data:
\begin{align*}
\wvarphi=-\pn \wu,\quad \quad \wpsi =s\gamma \wu,
\end{align*}
with $\gamma$ denoting the trace operator onto $\Gamma$.

The related integral operators on the boundary are denoted by
\begin{alignat}{2}
	V(s)\varphi(x)&=\int_{\Gamma}^{}\dfrac{1}{4\pi\left| x-y \right|}e^{-s\left|x-y\right|}\varphi(y) d\Gamma_y \quad \quad &&x\in \Gamma,\\
	K(s)\psi(x)&=\int_{\Gamma}^{}\left(\pnx\dfrac{1}{4\pi\left| x-y \right|}e^{-s\left|x-y\right|}\right)\psi(y) d\Gamma_y \quad \quad &&x\in \Gamma,\\
	K^T(s)\varphi(x)&=\pnx\int_{\Gamma}^{}\dfrac{1}{4\pi\left| x-y \right|}e^{-s\left|x-y\right|}\varphi(y) d\Gamma_y \quad \quad &&x\in \Gamma,\\
	W(s)\psi(x)&=-\pnx\int_{\Gamma}^{}\left(\pnx\dfrac{1}{4\pi\left| x-y \right|}e^{-s\left|x-y\right|}\right)\psi(y) d\Gamma_y \quad \quad &&x\in \Gamma.
\end{alignat}
They are bounded linear operators on the following spaces,
\begin{alignat*}{2}
	V(s) &\colon H^{-1/2}(\Gamma)\rightarrow H^{1/2}(\Gamma),\quad
	K(s) &&\colon H^{1/2}(\Gamma)\rightarrow H^{1/2}(\Gamma),\\
	K^T(s) &\colon H^{-1/2}(\Gamma)\rightarrow H^{-1/2}(\Gamma),
\ \	W(s) &&\colon H^{1/2}(\Gamma)\rightarrow H^{-1/2}(\Gamma).
\end{alignat*}
They are bounded, for $\text{Re} \ s \ge \sigma >0$, by $C_\sigma\abs{s}^{\kappa}$ with $\kappa=1,\,3/2,\,3/2,\,2$, respectively \LB{and a constant $C_\sigma$ growing at most polynomially as $\sigma \rightarrow 0$}; see
\cite{BH86a,BH86b}.

With the boundary operators we form the {\it Calder\'on operator}
\begin{equation}\label{calderon}
B(s)  =\begin{pmatrix}
s V(s) & K(s) \\
-K^T(s) & s^{-1}W(s)
\end{pmatrix}
\ : \ H^{-1/2}(\Gamma) \times H^{1/2}(\Gamma) \to H^{1/2}(\Gamma) \times H^{-1/2}(\Gamma).
\end{equation}
 From the mapping properties of the individual boundary integral operators we conclude that $B(s)$ is bounded as
 \begin{equation}\label{B-bound}
 \|B(s)\|_{H^{1/2}(\Gamma) \times H^{-1/2}(\Gamma) \leftarrow H^{-1/2}(\Gamma) \times H^{1/2}(\Gamma)} \leq C_\sigma |s|^2
 \quad\ \text{ for \ $\Re s \geq \sigma > 0$,}
 \end{equation}
\LB{with $C_\sigma$ growing polynomially in $\sigma^{-1}$ as $\sigma \rightarrow 0$.}
  In the following, we denote the anti-duality  between $H^{-1/2}(\Gamma) \times H^{1/2}(\Gamma)$ and $H^{1/2}(\Gamma) \times H^{-1/2}(\Gamma)$ by $\langle \cdot, \cdot \rangle_\Gamma$. We have the following important coercivity property.

\begin{lemma}\label{lemma:B-pos}
\cite[Lemma 3.1]{BLS15}
There exists $\beta > 0$ so that the Calder\'on operator \eqref{calderon} satisfies
\[
\Re \left \langle
  \begin{pmatrix}
      \varphi \\ \psi
  \end{pmatrix}, B(s) 
  \begin{pmatrix}
      \varphi \\ \psi
  \end{pmatrix}
\right \rangle_\Gamma
 \geq \beta\, \min(1,|s|^2) \, \Re s \, \left(\|s^{-1}\varphi\|^2_{H^{-1/2}(\Gamma)} + \|s^{-1}\psi\|^2_{H^{1/2}(\Gamma)}\right)
\]
for  $\Re s > 0$ and for all $\varphi \in H^{-1/2}(\Gamma)$ and $\psi \in H^{1/2}(\Gamma)$.
\end{lemma}

\subsection{Time-harmonic scattering from generalized impedance boundary conditions}
\label{subsec:time-harmonic}
\LB{We start with the time-harmonic formulation of the scattering problem \eqref{eq:scat_wave}
\[
s^2 \wuscat -\Delta \wuscat = 0\quad \quad\quad \quad\quad \quad \text{ in } \Omega.
\]
with the boundary condition
 \begin{align}\label{Generalbc_s}
 \pn \wuscat-F(s)s\wuscat = \wginc \quad\quad \quad \text{on } \Gamma
 \end{align}
where $\wginc = F(s)s\wuinc-\pn\wuinc.$}
Following \cite{BR18} (see also \cite{BL19}), we start from the observation that every solution $\wuscat\in H^1(\Omega)$ 
satisfies the identity
\begin{align}\label{Bimpsys}
\Bimp (s)
\begin{pmatrix}
-\pn \wuscat\\
 s\gamma\wuscat
\end{pmatrix}
=
\begin{pmatrix}
0\\
-\pn \wuscat
\end{pmatrix}, \qquad
\text{where}\quad \Bimp(s)=B(s) + \begin{pmatrix}
0 &-\tfrac{1}{2}I \\
\tfrac{1}{2}I & 0
\end{pmatrix}.
\end{align}
This is a consequence of Kirchhoff's representation theorem and the jump conditions of the potential operators. 
Inserting the boundary condition \eqref{Generalbc_s} into the right-hand side of (\ref{Bimpsys}) gives 
\begin{align*}
\Bimp(s)
\begin{pmatrix}
\wvarphi\\
\wpsi
\end{pmatrix}
=
\begin{pmatrix}
0\\
-F(s)\wpsi-\wginc
\end{pmatrix},
\end{align*}
where 
\begin{align*}
\wvarphi=-\pn \wuscat,\quad \quad \wpsi =s\gamma \wuscat.
\end{align*}
We rearrange in a way that all terms containing $\wpsi$ only appear on the left-hand side to arrive at the equation
\begin{align}\label{Asoperator}
A(s)\begin{pmatrix}
\wvarphi\\
\wpsi
\end{pmatrix}
=
\begin{pmatrix}
0\\
-\wginc
\end{pmatrix},
\quad\ \text{ where}\quad A(s) = \Bimp(s) 
+\begin{pmatrix}
0 &0 \\
0 & F(s)
\end{pmatrix}.
\end{align}
%
%
%
%

 The operator $A(s)$ inherits important properties of the Calder\'on operator $B(s)$. The following two lemmas collect bounds and coercivity results of $A(s)$ for  $F(s)$ polynomially bounded and of positive type.
\begin{lemma}[Boundedness]\label{lmbound}
If $F(s):X\rightarrow X'$ satisfies the bound \eqref{F-bound}  with $\mu\le 2$, then the corresponding operator $A(s)$, defined in \eqref{Asoperator}, is a bounded linear operator
$$
A(s)
\colon H^{-1/2}(\Gamma)\times X \rightarrow
H^{1/2}(\Gamma)\times X', \qquad  \Re s >0,
$$
which is bounded by $C_\sigma |s|^2$ for $\Re s \ge \sigma >0$.
\end{lemma}

This follows readily from the bounds of $B(s)$ and $F(s)$, noting that $X\subset H^{1/2}(\Gamma)$.
$A(s)$ also inherits a coercivity property from $B(s)$, which is formulated in the following Lemma.
\begin{lemma}[Coercivity]\label{lmcoerc}
Let $F(s):X\rightarrow X'$ satisfy the positivity condition \eqref{F-pos}. Then, for every $\sigma>\LB{\sigma_0 \geq 0}$ there exists $\alpha_\sigma>0$ such that for  $\Re s\ge\sigma$,
\emph{\begin{align*}
\Re \left\langle
\begin{pmatrix}
\wvarphi\\
\wpsi
\end{pmatrix},
A(s)
\begin{pmatrix}
\wvarphi\\
\wpsi
\end{pmatrix}
\right\rangle
\ge  \alpha_\sigma \left(\norm{
s^{-1}\wvarphi}^2_{H^{-1/2}(\Gamma)}
+\norm{
s^{-1}\wpsi}^2_{X} \quad\right),
\end{align*}}
where the anti-duality on the left-hand side is that between  $H^{-1/2}(\Gamma)\times X$ and $H^{1/2}(\Gamma)\times X'$.
\end{lemma}

\begin{proof} By Lemma~\ref{lemma:B-pos} and condition \eqref{F-pos}, we have for $\Re s \ge \sigma >0$
\begin{align*}
&\Re \left\langle
\begin{pmatrix}
\wvarphi\\
\wpsi
\end{pmatrix},
\left(\Bimp(s)
+
\begin{pmatrix}
0 &0 \\
0 & F(s)
\end{pmatrix}
\right)
\begin{pmatrix}
\wvarphi\\
\wpsi
\end{pmatrix}
\right\rangle
\\
&=
\Re \left\langle
\begin{pmatrix}
\wvarphi\\
\wpsi
\end{pmatrix},
B(s)
\begin{pmatrix}
\wvarphi\\
\wpsi
\end{pmatrix}
\right\rangle
+\text{Re }\langle \wpsi, F(s)\wpsi \rangle
\\ &\ge
\beta \min(1,\left|s\right|^2)
\,\sigma
\left(\norm{s^{-1}\wvarphi}^2_{H^{-1/2}(\Gamma)}+
\norm{s^{-1}\wpsi}^2_{H^{1/2}(\Gamma)}
\right)+c_\sigma \snorm{s^{-1}\wpsi}^2_X
\\ &\ge
 \alpha_\sigma \left(\norm{
s^{-1}\wvarphi}^2_{H^{-1/2}(\Gamma)}
+\norm{
s^{-1}\wpsi}^2_{H^{1/2}(\Gamma)}
+\snorm{
s^{-1}\wpsi}^2_{X}\right)
\\ &=
 \alpha_\sigma \left(\norm{
s^{-1}\wvarphi}^2_{H^{-1/2}(\Gamma)}
+\norm{
s^{-1}\wpsi}^2_{X}\right)
,
\end{align*}
where $\alpha_\sigma = \min \bigl( \beta \min(1,\sigma^2)
\,\sigma, c_\sigma\bigr)>0$.
\end{proof}\ \\

A direct consequence of the coercivity is the bound of the inverse, for $\Re s \ge \sigma > \LB{\sigma_0 \geq} 0$,
\begin{align}\label{boundAm1}
\norm{A^{-1}(s)}_{H^{-1/2}(\Gamma)\times X \leftarrow H^{1/2}(\Gamma)\times X'}\le C_\sigma \abs{s}^2.
\end{align}
Conversely, with the unique solution $(\wvarphi,\wpsi)$ of \eqref{Asoperator}, we use the representation formula to construct
\begin{equation}\label{rep-uscat}
\wuscat = S(s) \wvarphi + D(s)s^{-1} \wpsi,
\end{equation}
which is a solution to the Helmholtz equation with $\wvarphi=-\pn \wuscat$ and $\wpsi =s\gamma \wuscat$
that, by its very construction, satisfies the boundary condition
\eqref{Generalbc_s}.

Collecting the arguments in this subsection, we thus obtain the following well-posedness result for the time-harmonic scattering problem
\begin{equation}\label{bvp-gibc-s}
\begin{aligned}
s^2 \wuscat -\Delta \wuscat &= 0 \qquad\text{in }\Omega
\\[1mm]
 \pn \wuscat  - F(s)s\wuscat &= \wginc 
 \quad\ \text{on } \Gamma.
\end{aligned}
\end{equation}

\begin{proposition}[Well-posedness]\label{freq_ana}
Let $\Re s >\LB{\sigma_0 \geq }0$, $F(s)$ satisfy \eqref{F-bound} and \eqref{F-pos}, and let $\wginc \in X'$.
Then, the time-harmonic scattering problem \eqref{bvp-gibc-s} has a unique solution $\wuscat\in H^1(\Omega)$. This solution is given by the representation formula \eqref{rep-uscat}, where $(\wvarphi,\wpsi)\in H^{-1/2}(\Gamma)\times X$ is the unique solution of the boundary system \eqref{Asoperator}. Moreover, $\wvarphi=-\pn \wuscat$ and $ \wpsi =s\gamma \wuscat$.
The norms of $\wuscat$ and $(\wvarphi,\wpsi)$ are bounded polynomially in $|s|$ in terms of the norm of 
${\widehat g}^\text{inc}$.
 \qed
\end{proposition}

We further note that, with the same proof as that of Proposition 5.2 in \cite{BL19}, we obtain the $H^1$ bound
\begin{equation} \label{uscat-norm}
\norm{\wuscat }_{H^1(\Omega)} \le C_\sigma\, |s|^{5/2} \norm{{\widehat g}^\text{inc}}_{X'}, \qquad \Re s \ge \sigma> \LB{\sigma_0 \geq} 0.
\end{equation}
For $x\in\Omega$ with dist$(x,\Gamma)\ge\delta>0$, the bounds (5.12)--(5.13) in \cite{BL19} (which are based on Lemma 6 in \cite{BLM11}) together with the bound \eqref{boundAm1} yield the pointwise bound
\begin{equation}  \label{uscat-x}
\abs{\wuscat (x)} \le C_{\sigma,\delta}\, |s|^{3} \norm{\wginc}_{X'}, \qquad \Re s \ge \sigma> \LB{\sigma_0 \geq} 0.
\end{equation}

%
%

\subsection{Time-dependent scattering from generalized impedance boundary conditions}
\label{subsec:time-dep}
The above construction in the frequency domain extends to the time domain in a straightforward way, using the notation and results for temporal convolutions described in Section~\ref{subsec:gibc} (ii).
We start from the time-dependent version of \eqref{Bimpsys}, which holds for solutions $\uscat$ of the wave equation with
$\uscat \in H^r_0(0,T;H^1_\text{loc}(\Omega))$ for some real $r\ge 0$,
\begin{equation}\label{Bimpsys-t}
\Bimp (\pt)
\begin{pmatrix}
-\pn \uscat\\
 \pt\gamma \uscat
\end{pmatrix}
=
\begin{pmatrix}
0\\
-\pn \uscat
\end{pmatrix}.
\end{equation}
This leads us to the time-dependent version of \eqref{Asoperator},
\begin{align}\label{Asoperator-t}
A(\pt)\begin{pmatrix}
\varphi\\
\psi
\end{pmatrix}
=
\begin{pmatrix}
0\\
-\ginc
\end{pmatrix}.
\end{align}
We note that with $Z(s)=A(s)^{-1}$, which is bounded by \eqref{boundAm1}, the solution to this convolution equation is given as the convolution
\begin{align}\label{Zsoperator-t}
\begin{pmatrix}
\varphi\\
\psi
\end{pmatrix}
= Z(\pt)
\begin{pmatrix}
0\\
-\ginc
\end{pmatrix}.
\end{align}
We then use the Kirchhoff representation formula to construct 
\begin{equation}\label{rep-uscat-t}
\uscat = S(\pt) \varphi + D(\pt)\pt^{-1} \psi,
\end{equation}
which is a solution to the wave equation with the boundary data $\varphi=-\pn \uscat$ and $\psi=\pt\gamma \uscat$ and which satisfies the generalized impedance boundary condition.

With the arguments presented in the course of this section, we thus obtain an analogue of Proposition~\ref{freq_ana} for the time-dependent scattering problem
\eqref{eq:scat_wave}.
As the main interest in this paper lies in the numerical approximation, we do not carry out the details of the proof.

\begin{proposition}[Well-posedness]\label{t_ana} Let $F$ satisfy the polynomial bound \eqref{F-bound} and the positivity condition \eqref{F-pos} \LB{and let $\ginc \in H^{r}_0(0,T; X')$, for  $r\in\R$. Then  the time-dependent scattering problem \eqref{eq:scat_wave} has a unique solution $\uscat\in H^{r-5/2}_0(0,T;H^1(\Omega))$. This solution is given by the representation formula \eqref{rep-uscat-t}, where $(\varphi,\psi)\in H^{r-2}_0(0,T;H^{-1/2}(\Gamma)\times X)$ is the unique solution of the boundary system \eqref{Asoperator-t}. Moreover, $\varphi=-\pn \uscat$ and $ \psi =\pt\gamma \uscat$. 
The corresponding norms of $\uscat$ and $(\varphi,\psi)$ are bounded in terms of the norm of 
$\ginc  \in H^r_0(0,T; X')$.}
\end{proposition}

\LB{The smoothness requirements in the above proposition follow from bounds \eqref{uscat-norm} and \eqref{boundAm1} via the Plancherel formula argument given in Section~\ref{subsec:gibc}. The bound \eqref{uscat-x} further implies that
for $x\in\Omega$ with dist$(x,\Gamma)\ge\delta>0$, there is the pointwise bound
\begin{equation}  \label{uscat-x-t}
\norm{\uscat (x,\cdot)}_{H^{r-3}_0(0,T;\mathbb{R})} \le C_{T,\delta}\, \,\norm{g^\text{inc}}_{H^{r}_0(0,T;X')}.
\end{equation}}

\section{Semi-discretization in time by convolution quadrature}
\label{sec:cq}

\subsection{Convolution quadrature}
We consider convolution quadrature constructed from the Laplace transform of the convolution kernel and an A-stable linear multistep method, as studied in \cite{L94}. Given the Laplace transform $K(s)$, $\Re s \ge \sigma$, of a (distributional) convolution kernel $k(t)$, $t\ge 0$,
and a function $g(t)$, $0\le t \le T$,
we approximate the convolution $K(\partial_t)g=k*g$ by a discrete convolution with stepsize $\tau>0\,$:
\begin{align*}
	\left( K(\partial_t^\tau)g\right) (t):= \sum_{j\ge 0} \omega_j \, g(t-j\tau),
\end{align*}
which is defined for $0\le t \le T$, but usually considered only on the grid $t_n=n\tau$, so that only grid values of $g$ are required.
The convolution quadrature weights are defined as the coefficients of the generating power series
\begin{align*}
	\sum_{j=0}^{\infty}\omega_j \, \zeta^j:=K\left(\dfrac{\delta(\zeta)}{\tau}\right), 
\end{align*}
where we choose $\delta(\zeta)$ as the generating polynomial of the $p$th order backward differentiation formula (BDF),
$$
\delta(\zeta)= \sum_{\ell=0}^{p}\frac1\ell (1-\zeta)^\ell, \qquad p=1,2.
$$ 
We assume $p\le 2$, because only for $p\le 2$ the method is A-stable, i.e. $\Re \delta(\zeta)\ge 0$ for $|\zeta|\le 1$.

An important property of this convolution quadrature is that there is still an operational calculus: just as $K(\pt)L(\pt)g=(KL)(\pt)g$, 
we also have the composition rule
$$
K(\pt^\tau)L(\pt^\tau)g=(KL)(\pt^\tau)g.
$$ 
In particular, if $A(s)$ is an invertible linear operator for $\Re s \ge \sigma$, then
the solution of the convolution equation $A(\pt)\phi = g$ is given by $\phi=A^{-1}(\pt)g$, and the solution of the convolution quadrature approximation to the convolution equation, $A(\pt^\tau)\phi = g$, is given by $\phi=A^{-1}(\pt^\tau)g$, i.e., by the convolution quadrature for the Laplace transform $A^{-1}(s)$. This will be a key observation in the following.

As is shown in \cite{L94}, the convolution quadrature based on the $p$th order BDF method ($p=1,2$) yields an $O(\tau^p)$ approximation provided that $g$ together with its extension by $0$ to $t<0$ is sufficiently regular. As we will make repeated use of this result, we give a precise formulation for the convenience of the reader.

\begin{lemma}[\cite{L94}, Theorem 3.2] \label{lem:cq-error}
Let $X$ and $Y$ be complex Hilbert spaces, and let $K(s): X \to Y$, for $\Re s>0$, be an analytic family of linear operators bounded by 
\begin{equation}\label{K-bound}
\| K(s) \|_{Y \leftarrow X} \le M_\sigma \, |s|^\mu, \qquad \Re s \ge \sigma >\LB{\sigma_0 \geq} 0.
\end{equation}
Let $g \in H^{r+1/2}_0(0,T;X)$ (see Section~\ref{subsec:gibc} for the notation), where 
$$
r>p \quad\text{ and }\quad r-\mu -1> p.
$$
Then, the error of the convolution quadrature based on the $p$th order BDF method ($p\le 2$) is bounded by
$$
\| K(\pt^\tau) g (t) - K(\pt) g (t) \|_Y \le C_T \, \tau^p \, \| g \|_{H^{r+1/2}_0(0,T;X)}\ , \qquad 0\le t \le T,
$$
where $C_T$ depends only on $\mu$ and $r$ and $p$ and is proportional to $\LB{e^{\sigma_0T}M_{\sigma_0+1/T}}$.
\end{lemma}

We will also need an error bound in the $H^m_0(0,T;Y)$ norm, for some $m\ge 0$. 

\begin{lemma} \label{lem:cq-error-Hm}
Let $m\ge 0$. In the situation of Lemma~\ref{lem:cq-error}, let $g \in H^{m+r}_0(0,T;X)$, where again
$$
r>p \quad\text{ and }\quad r-\mu -1> p.
$$
Then, the error of the convolution quadrature based on the $p$th order BDF method ($p\le 2$) is bounded by
$$
\| K(\pt^\tau) g  - K(\pt) g  \|_{H^m_0(0,T;Y)} \le C_T \, \tau^p \, \| g \|_{H^{m+r}_0(0,T;X)}\ , \qquad 0\le t \le T,
$$
where $C_T$ depends only on $\mu$ and $r$ and $p$ and is proportional to \LB{$e^{\sigma_0 T} M_{\sigma_0+1/T}$.}
\end{lemma}

\begin{proof} The proof uses arguments and estimates from the proofs in Section~3 of \cite{L94}. Using the Plancherel formula,
$K(\pt^\tau) g  - K(\pt) g$ is bounded in the $H^m_0(0,T;Y)$ norm by  $\| g \|_{H^{m+r}_0(0,T;X)}$ times a constant multiple of the factor
$$
\sup_{\Re s=\sigma} \norm{\left( K\Bigl( \frac{\delta(e^{-s\tau})}{\tau} \Bigr) - K(s) \right) s^{-r} }_{Y\leftarrow X}.
$$
In Section~3 of \cite{L94}, this expression is bounded by $c_\sigma \tau^p$ under the stated conditions on $r$. This yields the result.
\end{proof}

We will also use the following stability result.

\begin{lemma} \label{lem:cq-stab}
Let $K(s)$ be as in Lemma~\ref{lem:cq-error}, with $\mu\ge 0$ in \eqref{K-bound}. Let $g \in H^{r+1/2}_0(0,T;X)$, where 
$$
r>\mu.
$$
Then, the result of the convolution quadrature based on the $p$th order BDF method ($p\le 2$) is bounded by
$$
\| K(\pt^\tau) g (t)  \|_Y \le C_T \,  \| g \|_{H^{r+1/2}_0(0,T;X)}\ , \qquad 0\le t \le T,
$$
where $C_T$ depends on $T$ but is independent of $\tau$ and $t$.
\end{lemma}

\begin{proof} We note that $K(\pt^\tau) g=K_\tau(\pt)g$ with $K_\tau(s)=K(\delta(e^{-s\tau})/\tau)$. 
Using the fact that $\delta(1)=0$ and that there is no other zero of $\delta$ in the closed unit disk, 
we find that for some $c$ independent of $\tau\le T$,
$$
\abs{\frac{\delta(e^{-s\tau})}{\tau}} \le c |s| \quad\ \text{ for } \ \Re s \ge 0.
$$
It then follows that, with  ${\widetilde M}_\sigma= c^\mu M_\sigma$,
$$
\| K_\tau(s) \|_{Y \leftarrow X} \le {\widetilde M}_\sigma \, |s|^\mu, \qquad \Re s \ge \sigma >0.
$$
By the same argument as in \eqref{sobolev-bound}, now for $K_\tau$ in place of $F$, we obtain the linear operator, bounded uniformly in $\tau$,
$$
K_\tau(\pt) \,:\, H^{r+1/2}_0(0,T;X) \to H^{r-\mu+1/2}_0(0,T;Y) .
$$
Since $r-\mu>0$, we have that $H^{r-\mu+1/2}_0(0,T;Y)$ is continuously embedded in $C([0,T],Y)$, and hence the results follows.
\end{proof}

%
%
\subsection{Convolution quadrature for generalized impedance boundary conditions}

Applying the convolution quadrature method to \eqref{Asoperator-t} yields the discrete convolution equation
\begin{align}\label{AoperatorCQ}
A(\pt^\tau)\begin{pmatrix}
\varphi^\tau\\
\psi^\tau
\end{pmatrix}=
\begin{pmatrix}
0\\
-F(\pt^\tau)\pt^\tau \gamma u^\text{inc}+\pn u^\text{inc}
\end{pmatrix}.
\end{align}
We have the following error bound pointwise in time.

\begin{theorem}\label{thm:err-cq}
Let $F$ satisfy the polynomial bound (2.3) and the positivity
condition (2.4). Let $A(s)$ be the corresponding operator from \eqref{Asoperator} and let 
$\gamma u^\text{inc}\in H^r_0(0,T;X)$ and $\pn u^\text{inc}\in H^r_0(0,T;X')$ for a sufficiently large $r$.
Then, the error of the temporal semi-discretization \eqref{AoperatorCQ} by convolution quadrature based on a BDF method of order $p\le 2$ is bounded by
\begin{align*}
\max_{0\le n \le N}
\norm{\begin{pmatrix}
\varphi^\tau(t_n)-\varphi(t_n)\\
\psi^\tau(t_n)-\psi(t_n)
\end{pmatrix}}_{H^{-1/2}(\Gamma)\times X}
\le C_T \,\tau^p \left(\norm{\gamma u^\text{inc}}_{H^r_0(0,T;X)} + \norm{\pn u^\text{inc}}_{H^r_0(0,T;X')}\right),
\end{align*}
 where $C_T$ depends only on the boundary $\Gamma$ and on the final time $T\ge N\tau$.
\end{theorem}
\begin{proof} With $Z(s)=A(s)^{-1}$,  we have the convolution \eqref{Zsoperator-t}.
The convolution quadrature method \eqref{AoperatorCQ} is equivalent to the convolution quadrature approximation of \eqref{Zsoperator-t}, that is,
\begin{align}\label{Zsoperator-t-tau}
\begin{pmatrix}
\varphi^\tau\\
\psi^\tau
\end{pmatrix}
= Z(\pt^\tau)
\begin{pmatrix}
0\\
-F(\pt^\tau)\pt^\tau\gamma u^\text{inc}+\pn u^\text{inc}
\end{pmatrix}.
\end{align}
In view of the bound \eqref{boundAm1} of $Z(s)$ and the bound \eqref{F-bound} of $F(s)$ for $\Re s>0$, the result follows from the general convolution quadrature error bound given in Lemma~\ref{lem:cq-error}.
\end{proof}\ \\

The convolution quadrature approximation to the scattered wave is then obtained by discretizing the Kirchhoff representation formula \eqref{rep-uscat-t}:
\begin{equation}\label{rep-uscat-t-tau}
\uscatau = S(\pt^\tau) \varphi^\tau + D(\pt^\tau)(\pt^\tau)^{-1} \psi^\tau.
\end{equation}
We have the following optimal-order error bounds.

\begin{theorem}\label{thm:semidisconv-uscat}
Let $F$ satisfy the polynomial bound (2.3) and the positivity
condition (2.4). Let $A(s)$ be the corresponding operator from \eqref{Asoperator} and let 
$\gamma u^\text{inc}\in H^r_0(0,T;X)$ and $\pn u^\text{inc}\in H^r_0(0,T;X')$ for a sufficiently large $r$.
Then, the error of the CQ semi-discretization \eqref{rep-uscat-t-tau} with \eqref{AoperatorCQ}, based on a BDF method of order $p\le 2$, is bounded by
\begin{align*}
\max_{0\le n \le N}
\norm{\uscatau(t_n)- \uscat(t_n) }_{H^{1}(\Omega)}
\le C_T\,\tau^p \left(\norm{\gamma u^\text{inc}}_{H^r_0(0,T;X)} + \norm{\pn u^\text{inc}}_{H^r_0(0,T;X')}\right),
\end{align*}
 where $C_T$ depends only on the boundary $\Gamma$ and on the final time $T\ge N\tau$.
Furthermore, for $x\in\Omega$ with dist$(x,\Gamma)\ge \delta>0$ we have the pointwise error bound
\begin{align*}
\max_{0\le n \le N}
\abs{\uscatau(x,t_n)- \uscat(x,t_n) }
\le C_{\delta,T}\,\tau^p \left(\norm{\gamma u^\text{inc}}_{H^r_0(0,T;X)} + \norm{\pn u^\text{inc}}_{H^r_0(0,T;X')}\right),
\end{align*}
 where $C_{\delta,T}$ depends on the boundary $\Gamma$, on $\delta$ and  on the final time $T$.
\end{theorem}

\begin{proof} We note that by combining \eqref{rep-uscat-t-tau} and \eqref{Zsoperator-t-tau} and setting $U(s):= \bigl( S(s) , D(s)s^{-1} \bigr) Z(s)$, we have the convolution quadrature
\begin{equation}\label{rep-uscat-t-tau-cq}
\uscatau = U(\pt^\tau)
\begin{pmatrix}
0\\
-F(\pt^\tau)\pt^\tau\gamma u^\text{inc}+\pn u^\text{inc}
\end{pmatrix}.
\end{equation}
With the bounds for $U(s)$ given by \eqref{uscat-norm} and \eqref{uscat-x} (note that $\wuscat=U(s)\LB{
  \begin{pmatrix}
0\\-\wginc
  \end{pmatrix}}$), the result follows again from Lemma~\ref{lem:cq-error}.
\end{proof}

\begin{remark} Here we have only considered convolution quadrature based on multistep methods. For Runge--Kutta convolution quadrature, similar results can be obtained using the error bounds of \cite{BLM11}. This gives methods that are of higher than second order in time.
\end{remark}

\begin{remark}[Dependence of constants on $T$ and $\varepsilon$]
\LB{Throughout,  the constants in the estimates such as $C_T$ in Theorem~\ref{thm:semidisconv-uscat} are allowed to grow as $e^{\sigma T}$ as $T$ increases for an arbitrary  $\sigma \geq \sigma_0$; see Lemma~\ref{lem:cq-error} and \ref{lem:cq-error-Hm}. Further, the constants grow at most polynomially as $\sigma \rightarrow 0$; see Remark~\ref{rem:gibc}, \eqref{B-bound} and Lemma~\ref{lemma:B-pos}. Hence in the case of boundary conditions (A), (B1), and (C) we can let $\sigma = 1/T$ and obtain bounds that grow only polynomially with $T$. For (B2), we set $\sigma = \sigma_0+1/T = \varepsilon^2 \|\mathcal{H}\|_{L^\infty(\Gamma)}^2+1/T
$ to obtain bounds with no visible exponential increase for long computational times $T \propto \varepsilon^{-2}$.}

\LB{
For (A) \JN{and (B)} the constants in the error estimates are independent of the small parameter $\varepsilon$. \JN{However, since the norm of $X$ corresponding to (B) depends inversely on $\varepsilon^{1/2}$, the error bounds in Theorem~\ref{thm:err-cq} and Theorem~\ref{thm:semidisconv-uscat} grow as $\varepsilon^{-1/2}$ in the case of the boundary conditions (B1) and (B2).}
}
\end{remark}

%

\section{Semi-discretization in space by boundary elements}
\label{sec:bem}
We consider a family of regular triangulations of the boundary $\Gamma$ with maximal meshwidths $h\to 0$ and boundary element spaces
$$
\Phi_h \subset H^{-1/2}(\Gamma),\quad
\Psi_h \subset X \subset H^{1/2}(\Gamma),
$$
and actually $\Phi_h \subset L^2(\Gamma)$.
With the spaces $\mX=H^{-1/2}(\Gamma)\times X$ and $\mX_h = \Phi_h \times \Psi_h\subset \mX$, we are thus in the more abstract situation of the following subsection.

\subsection{Galerkin semi-discretization: stability and quasi-optimality}

We have Hilbert spaces $\mX\subset \mH=\mH' \subset \mX'$ with dense and continuous inclusions, and we have an analytic family of
operators $A(s):\mX\to \mX'$, bounded by $M_\sigma |s|^2$ for $\Re s\ge \sigma >0$,
with a coercivity estimate (with $\alpha_\sigma>0$)
\begin{align*}
\text{Re } \big\langle
\widehat{\phi}, A(s)\widehat{\phi} \big\rangle
\ge \alpha_\sigma \| s^{-1}\widehat{\phi} \|_\mX^2
\qquad \text{for all }\ \widehat{\phi}\in \mX, \ \Re s\ge \sigma >0.
\end{align*}
The convolution equation
$ 
A(\pt)\phi=g
$ 
is approximated on the finite dimensional subspace $\mX_h\subset \mX$ by the Galerkin semi-discretization
\begin{align}
\label{galerkin}
\langle
\chi_h, (A(\pt)\phi_h)(t)
\rangle=
\langle
\chi_h, g(t)
\rangle \qquad \text{for all }\ \chi_h \in \mX_h \ \text{ and a.e. } t\in [0,T].
\end{align}
We let $\Pi_h$ be the $\mX$-orthogonal projection from $\mX$ onto $\mX_h$ so that for $\phi \in \mX$, $\Pi_h \phi$ is the best approximation to $\phi$ in $\mX_h$ with respect to the norm of $\mX$. The following lemma shows the stability of the Galerkin semi-discretization and yields an error bound reminiscent of C\'ea's lemma for elliptic problems.

\begin{lemma}\label{lem:galerkin} Let $m\ge 0$.
Under the above assumptions and provided that $g\in H_0^{m+2}(0,T;\mX')$, the  Galerkin semi-discretization \eqref{galerkin} has a unique solution $\phi \in H_0^{m}(0,T;\mX)$,  bounded by
\begin{align*}
\norm{\phi_h}_{H_0^{m}(0,T;\mX)}\le c_T\norm{g}_{H_0^{m+2}(0,T;\mX')}.
\end{align*}
The  error is bounded by
\begin{align*}
\norm{\phi_h(t)-\phi(t)}_{H_0^{m}(0,T;\mX)}
\le C_T\norm{\Pi_h\phi-\phi}_{H_0^{m+4}(0,T;\mX)}.
\end{align*}
The constants $c_T$ and $C_T$ are independent of $h$. They are inversely proportional to $\alpha_{1/T}$, and $C_T$  is additionally proportional to $M_{1/T}$.
\end{lemma}

\begin{proof} (a) We define $A_h(s):\mX_h \to \mX_h$ by setting, with the inner product $(\cdot,\cdot)$ of $\mH$,
$$
(\widehat\chi_h, A_h(s) \widehat\phi_h) = \langle \widehat\chi_h,A(s)\widehat\phi_h \rangle \qquad \forall \chi_h,\phi_h \in \mX_h.
$$
Let $P_h\colon \mX'\rightarrow \mX_h$ be the $\mH$-orthogonal projection, i.e., 
\begin{align*}
(\widehat\chi_h, P_h \widehat g)=\langle \widehat\chi_h,\widehat g\rangle \quad \forall \chi_h\in \mX_h\subset \mX, \ \widehat g \in \mX'.
\end{align*}

Then, $A_h(s)$ inherits the coercivity of $A(s)$ and is therefore invertible.
Testing the equation $A_h(s)\widehat\phi_h=P_h \widehat g$ with $\widehat\chi_h=  \widehat\phi_h$ and using the coercivity yields
\begin{align*}
&\alpha_\sigma\, \| s^{-1} \widehat\phi_h \|_\mX^2 \le \Re \langle \widehat\phi_h,A(s)\widehat\phi_h = 
\Re (\widehat\phi_h, A_h(s) \widehat\phi_h) 
= \Re (\widehat\phi_h, P_h \widehat g) = \Re \langle \widehat\phi_h,\widehat g\rangle
\le \| \widehat\phi_h \|_\mX \, \| \widehat g \|_{\mX'}
\end{align*}
so that $\widehat \phi_h=A_h(s)^{-1} P_h \widehat g$ is bounded by $\| \widehat \phi_h \|_\mX 
\le (1/\alpha_\sigma)\, |s|^2\,\| \widehat g \|_{\mX'}$. Hence,
\begin{equation}\label{Ah-inv-bound}
\| A_h(s)^{-1} P_h \|_{\mX \leftarrow \mX'} \le \frac 1{\alpha_\sigma}\, |s|^2\qquad\text{for }\ \Re s \ge \sigma >0.
\end{equation}
The Galerkin condition \eqref{galerkin} can be rewritten as
$$
(\chi_h, A_h(\pt)\phi_h) =
(\chi_h, P_h g) 
\qquad \text{for all }\ \chi_h \in \mX_h
$$
or equivalently as
$$
A_h(\pt)\phi_h = P_h g
$$
or again equivalently as
$$
\phi_h = A_h^{-1}(\pt)P_h g.
$$
In view of the bound \eqref{Ah-inv-bound}, the Plancherel formula argument of (ii) in Section~\ref{subsec:gibc} (for $\mu=2$) yields the stated bound of $\phi_h$.

(b) Adding $\Pi_h\phi$ on both sides of \eqref{galerkin} and rearranging yields
\begin{align*}
\left\langle
\chi_h, A(\pt)\Pi_h\phi\right\rangle=\langle\chi_h, A(\pt)\left(\Pi_h\phi
-\phi
\right)
\rangle+
\langle \chi_h,g\rangle \qquad \forall \chi_h\in X_h\subset X
\end{align*}
We denote the error and the defect by
\begin{alignat*}{2}
e_h&:=\phi_h-\Pi_h\phi \quad\quad\quad\quad  &&\text{ in }\mX,\\
d&:=A(\pt)\left(\Pi_h\phi 
-\phi
\right)&&\text{ in }\mX' .
\end{alignat*}
We obtain the error equation
\begin{align*}
A_h(\pt)e_h=P_h d,
\end{align*}
or rearranged
\begin{align*}
e_h=A_h^{-1}(\pt) P_h d.
\end{align*}
As in part (a) of this proof, the stated error bound then follows from \eqref{Ah-inv-bound} and the bound for~$A(s)$:
$$
\| e_h \|_{H_0^{m}(0,T;\mX)} \le c_T \| d \|_{H_0^{m+2}(0,T;\mX')} \le C_T\norm{\Pi_h\phi-\phi}_{H_0^{m+4}([0,T],\mX)}.
$$
This completes the proof.
\end{proof}

\pagebreak[3]

\subsection{Error bounds for the boundary element semi-discretization}
We now turn to error bounds in terms of the meshwidth $h$ for the problems (A)--(C) of the introduction.
We choose $\Phi_h$ to be a boundary element space with piecewise polynomial basis functions of degree $k-1$ (discontinuous if $k=1$ and continuous for $k>1$) and $\Psi_h$ to be a boundary element space of piecewise polynomials of degree $k\ge 1$ (globally continuous).

For $r\in [-1,1]$, let $\Pi_{\Phi_h}^r: H^r(\Gamma)\to \Phi_h$ and $\Pi_{\Psi_h}^r: H^r(\Gamma)\to \Psi_h$ denote the $H^r(\Gamma)$-orthogonal projection onto $\Phi_h$ and $\Psi_h$, respectively. For a regular family of triangulations, we then have the following standard bounds for the boundary element best-approximation error in $H^r(\Gamma)$:
\begin{equation}
\label{bem-approx}
\begin{aligned}
&\norm{\Pi_{\Phi_h}^r\eta-\eta}_{H^r(\Gamma)}\le C\,h^{k-r}\norm{\eta}_{H^{r+k-1}(\Gamma)},
\\
&\norm{\Pi_{\Psi_h}^r\eta-\eta}_{H^r(\Gamma)}\le C\,h^{k+1-r}\norm{\eta}_{H^{r+k}(\Gamma)}.
\end{aligned}
\end{equation}
We will also use a similar bound for the boundary element interpolation $I_h$ \bcl for $r=0,\frac12,1$: \ecl
\begin{equation}\label{bem-approx-ipol}
\norm{I_h\eta-\eta}_{H^r(\Gamma)}\le C\,h^{k+1-r}\norm{\eta}_{H^{r+k}(\Gamma)}.
\end{equation}

\noindent With these bounds, we obtain the following error bound for the space discretization errors.
We recall that $\mX=H^{-1/2}(\Gamma)\times X$, where $X\subset H^{1/2}(\Gamma)$ depends on the boundary condition (A), (B) or (C); see Remark~\ref{rem:gibc}.

\begin{theorem}\label{thm:err-bem} Let $m\ge 0$, and assume that the solution $(\varphi,\psi)$ of the time-dependent boundary integral equation \eqref{Asoperator-t} corresponding to the boundary conditions (A)--(C)
has the regularity $\varphi\in H^{m+4}_0(0,T;H^{k}(\Gamma))$ and $\psi\in H^{m+4}_0(0,T;H^{k+1}(\Gamma))$.
Then, the error of the boundary element spatial semi-discretization \eqref{galerkin} with \eqref{bem-approx}--\eqref{bem-approx-ipol}
is bounded by
\bcl
\begin{align*}
\norm{\begin{pmatrix}
\varphi_h-\varphi\\
\psi_h-\psi
\end{pmatrix}}_{H_0^{m}(0,T;H^{-1/2}(\Gamma)\times X)} \
&\le 
C \, ( h^{k+1/2}+\eps^{1/2} h^{k})
\qquad\qquad &\text{in case (A)},
\\
&\le 
C \, ( h^{k+1/2}+\eps^{-1/2} h^{k+1})
\qquad &\text{in case (B)},
\\[2mm]
&\le C \, h^{k+1/2}
\qquad\qquad\qquad\qquad &\text{in case (C)}.
\end{align*}
\ecl
The constant $C$ depends on the boundary $\Gamma$, on the final time $T$ and on the norms of the solution $(\varphi,\psi)$ in the spaces indicated above, but is independent of the meshwidth $h$ and, in case (A), of the small parameter $\eps$.

\end{theorem}
\begin{proof} The proof combines the error bound of Lemma~\ref{lem:galerkin} with the approximation estimates \eqref{bem-approx}--\eqref{bem-approx-ipol}.
For 
$$
\phi = \begin{pmatrix} \varphi \\ \psi \end{pmatrix} \in \mX = H^{-1/2}(\Gamma)\times X,
$$ 
the $\mX$-orthogonal projection $\Pi_h\phi$ equals
$$
\Pi_h\phi =  \begin{pmatrix} 
\Pi_{\Phi_h}^{-1/2}\varphi 
\\ 
\Pi_{X_h}\psi 
\end{pmatrix} ,
$$
where $\Pi_{X_h}$ is the $X$-orthogonal projection from $X$ onto $X_h$.

In case (C) we have simply $X=H^{1/2}(\Gamma)$, and then $\Pi_{X_h}\psi= \Pi_{\Psi_h}^{1/2}\psi$. Using \eqref{bem-approx} in Lemma~\ref{lem:galerkin} then yields the result.

In case (A) we have $X=H^1(\Gamma)$ equipped with the norm 
$\| \psi \|_X^2 = \| \psi \|_{H^{1/2}(\Gamma)}^2 + \eps \| \nabla\psi \|_{L^2(\Gamma)}^2$. We then find
\begin{align*}
&\norm{\Pi_{X_h}\psi-\psi}_{H_0^{m+4}(0,T;X)} 
= \norm{\pt^{m+4}(\Pi_{X_h}\psi-\psi)}_{L^2(0,T;X)}
= \norm{\Pi_{X_h}\pt^{m+4}\psi-\pt^{m+4}\psi}_{L^2(0,T;X)}
\\
&\le \norm{I_h\pt^{m+4}\psi-\pt^{m+4}\psi}_{L^2(0,T;X)}
\\
&
\le \biggl( \int_0^T \norm{I_h \pt^{m+4}\psi-\pt^{m+4}\psi}_{H^{1/2}(\Gamma)}^2
+ \eps \norm{I_h \pt^{m+4}\psi-\pt^{m+4}\psi}_{H^{1}(\Gamma)}^2 \biggr)^{1/2}.
\end{align*}
With \eqref{bem-approx-ipol} we thus obtain
$$
\norm{\Pi_{X_h}\psi-\psi}_{H_0^{m+4}(0,T;X)} \le C h^{k+1/2} + C\eps^{1/2} h^{k},
$$
where the constants depend on bounds of higher partial derivatives of $\varphi$ and $\psi$.

\bcl
For the case (B) we obtain similarly, with $X=H^{1/2}(\Gamma)$ equipped with the norm 
$\| \psi \|_X^2 = \| \psi \|_{H^{1/2}(\Gamma)}^2 + \eps^{-1} \| \psi \|_{L^2(\Gamma)}^2$, that
$$
\norm{\Pi_{X_h}\psi-\psi}_{H_0^{m+4}(0,T;X)} \le C h^{k+1/2} + C\eps^{-1/2} h^{k+1},
$$
where the constants depend on bounds of higher partial derivatives of $\varphi$ and $\psi$.
\ecl

The result then follows from Lemma~\ref{lem:galerkin}.
\end{proof}

%
%

\section{Full discretization by boundary elements and convolution quadrature}
\label{sec:cq-bem}

In this section we combine the error bounds of the temporal and spatial semi-discretizations for the full discretization of the system of time-dependent boundary integral equations \eqref{Asoperator-t} for the cases (A)--(C) of generalized impedance boundary conditions. Here, the approximation
$(\varphi_h^\tau,\psi_h^\tau)\approx (\varphi,\psi)$, which is in the boundary element space $\Phi_h\times\Psi_h$ pointwise in time, is determined by
\begin{align}\label{AoperatorCQ-BEM}
\left\langle \begin{pmatrix}
\chi_h\\
\eta_h
\end{pmatrix},
A(\pt^\tau)\begin{pmatrix}
\varphi^\tau_h\\
\psi^\tau_h
\end{pmatrix}
\right\rangle
=
\left\langle \begin{pmatrix}
\chi_h\\
\eta_h
\end{pmatrix},
\begin{pmatrix}
0\\
-F(\pt^\tau)\pt^\tau \gamma u^\text{inc}+\pn u^\text{inc}
\end{pmatrix}
\right\rangle
\quad\ \forall\, \chi_h\in\Phi_h, \,\eta_h\in \Psi_h.
\end{align}

\begin{theorem}\label{thm:err-cq-bem} Consider the full discretization \eqref{AoperatorCQ-BEM} by convolution quadrature based on a backward difference formula of order $p\le 2$ and by boundary elements of polynomial degree $k-1$ and $k\ge 1$ for the approximation of $\varphi$ and $\psi$, respectively.
Under the conditions of Theorems~\ref{thm:err-cq} and~\ref{thm:err-bem} (with $m\ge 1$), the error  is bounded by
\bcl
\begin{align*}
\max_{0\le n \le N}
\norm{\begin{pmatrix}
\varphi^\tau_h(t_n)-\varphi(t_n)\\
\psi^\tau_h(t_n)-\psi(t_n)
\end{pmatrix}}_{H^{-1/2}(\Gamma)\times X}
&\le 
C \, (\tau^p + h^{k+1/2}+\eps^{1/2} h^{k})
\qquad &\text{in case (A)},
\\
&\le 
C \, (\tau^p + h^{k+1/2}+\eps^{-1/2} h^{k+1})
\qquad &\text{in case (B)},
\\[2mm]
&\le C \, (\tau^p + h^{k+1/2})
\qquad &\text{in case (C)}.
\end{align*}
\ecl
The constant $C$ is independent $h$ and $\tau$ and $n$ with $t_n=n\tau\le T$, and of $\eps$ in case (A), but depends on $T$ and higher Sobolev norms of the incident wave $\gamma\uinc$ and $\pn\uinc$ and of the solution $(\varphi,\psi)$.
\end{theorem}

%
%
\begin{proof}
We split the error into
\begin{align*}
\begin{pmatrix}\varphi^{\tau}_h-\varphi\\
\psi^{\tau}_h-\psi
\end{pmatrix}
=
\begin{pmatrix}\varphi^{\tau}_h-\varphi_h\\
\psi^{\tau}_h-\psi_h
\end{pmatrix}
+
\begin{pmatrix}\varphi_h-\varphi\\ 
\psi_h-\psi
\end{pmatrix}
.
\end{align*}
The first term is given by 
\begin{align*}
\begin{pmatrix}\varphi^{\tau}_h-\varphi_h\\
\psi^{\tau}_h-\psi_h
\end{pmatrix}
=A_h^{-1}(\pt^\tau)P_h \begin{pmatrix}
0\\
-F(\pt^\tau)\pt^\tau \gamma u^\text{inc}+\pn u^\text{inc}
\end{pmatrix}
-A_h^{-1}(\pt)P_h \begin{pmatrix}
0\\
-F(\pt)\pt \gamma u^\text{inc}+\pn u^\text{inc}
\end{pmatrix}.
\end{align*}
In view of the stability bound \eqref{Ah-inv-bound}, the convolution quadrature error bound of Lemma~\ref{lem:cq-error} yields
$$
\max_{0\le n \le N}
\norm{\begin{pmatrix}
\varphi^\tau_h(t_n)-\varphi_h(t_n)\\
\psi^\tau_h(t_n)-\psi_h(t_n)
\end{pmatrix}}_{H^{-1/2}(\Gamma)\times X}
\le 
C \, \tau^p .
$$
The second term is the space discretization error, which is estimated by Theorem \ref{thm:err-bem}. Noting that
$H_0^{m}(0,T;\mX)$ is continuously embedded in $C([0,T],\mX)$ for $m>1/2$, we obtain the stated error bound pointwise in time.
\end{proof}

The  approximation to the scattered wave can then be obtained by  the discretized Kirchhoff representation formula \eqref{rep-uscat-t-tau}:
\begin{equation}\label{rep-uscat-t-tau-h}
\uscatau_h = S(\pt^\tau) \varphi^\tau_h + D(\pt^\tau)(\pt^\tau)^{-1} \psi^\tau_h.
\end{equation}
We have the following optimal-order error bounds.

\begin{theorem}\label{thm:err-cq-bem-uscat}
Under the conditions of Theorem~\ref{thm:err-cq-bem} and assuming sufficiently high regularity, the error in the approximation \eqref{rep-uscat-t-tau-h} to the scattered wave is bounded by
\bcl
\begin{align*}
\max_{0\le n \le N}
\norm{\uscatau_h(t_n)- \uscat(t_n) }_{H^{1}(\Omega)}
&\le 
C \, (\tau^p + h^{k+1/2}+\eps^{1/2}\, h^{k})
\qquad &\text{in case (A)},
\\
&\le 
C \, (\tau^p + h^{k+1/2}+\eps^{-1/2} h^{k+1})
\qquad &\text{in case (B)},
\\[2mm]
&\le C \, (\tau^p + h^{k+1/2})
\qquad\: &\text{in case (C)}.
\end{align*}
\ecl
The constant $C$ is independent $h$ and $\tau$ and $n$ with $t_n=n\tau\le T$, and of $\eps$ in case (A), but depends on $T$ and higher Sobolev norms of the incident wave $\gamma\uinc$ and $\pn\uinc$ and of the solution $(\varphi,\psi)$.

For $x\in\Omega$ with dist$(x,\Gamma)\ge \delta>0$, the pointwise error 
$\max_{0\le n \le N} \bigl| \uscatau_h(x,t_n)- \uscat(x,t_n) \bigr|$ 
has a bound of the same type, with a constant that depends additionally on $\delta$.
\end{theorem}

\begin{proof} We study the error of $S(\pt^\tau) \varphi^\tau_h$. The second term in the representation formula is estimated in an analogous way. We write the error as
\begin{equation}\label{S-err}
S(\pt^\tau) \varphi^\tau_h - S(\pt) \varphi = S(\pt^\tau) (\varphi^\tau_h -\varphi) + \bigl( S(\pt^\tau) \varphi - S(\pt) \varphi \bigr).
\end{equation}
The error $\varphi^\tau_h -\varphi$ can be bounded not only in the maximum norm in time, as we did in Theorem~\ref{thm:err-cq-bem}, but there is also a bound of the same order of approximation in the $H^2_0(0,T;H^{-1/2}(\Gamma))$ norm. This error bound is obtained by using Lemma~\ref{lem:cq-error-Hm} instead of Lemma~\ref{lem:cq-error}. By Lemma~\ref{lem:cq-stab} applied with $S(s)$ in the role of $K(s)$ (with $\mu=1$), we then have 
$$
\max_{0\le t \le T} \| S(\pt^\tau) (\varphi^\tau_h -\varphi)(t) \|_{H^1(\Omega)} \le C_0 \| \varphi^\tau_h -\varphi \|_{H^2_0(0,T;H^{-1/2}(\Gamma))} \le
C \, (\tau^p + h^{k+1/2}+\eps^{1/2}\, h^{k})
$$
in the case (A), \bcl with the term $\eps^{-1/2}\, h^{k+1}$ in case (B), and without the term $\eps^{1/2}\, h^{k}$ in the case  (C).\ecl
The other term in \eqref{S-err} is just a convolution quadrature error, and with Lemma~\ref{lem:cq-error} we obtain
$$
\| S(\pt^\tau) \varphi - S(\pt) \varphi \|_{H^1(\Omega)} \le C \, \tau^p.
$$
Using the same arguments for the error of $D(\pt^\tau)(\pt^\tau)^{-1} \psi^\tau_h$ and combining the error bounds yields the stated $H^1(\Omega)$-norm error bound for $\uscatau_h$. 

The pointwise error bound in $x\in \Omega$ is obtained in the same way,
using the pointwise operators defined by
$S_x(s)\varphi = (S(s)\varphi)(x)$ and $D_x(s)\psi = (D(s)\psi)(x)$. These are linear operators
$S_x(s):H^{-1/2}(\Gamma)\to \mathbb{C}$ and 
$D_x(s):H^{1/2}(\Gamma)\to \mathbb{C}$
that are bounded for $\Re s\ge\sigma >0$ and dist$(x,\Gamma)\ge \delta>0$ by $C_{\sigma,\delta} |s|$; see \cite{BLM11}, Lemma~6,
and \cite{BL19}, formulas (5.12)--(5.13). The same arguments then apply.
\end{proof}

\section{Numerical experiments}
\label{sec:num}
 All numerical experiments were conducted using continuous piecewise linear boundary element functions in space and convolution quadrature based on the second-order backward difference formula in time.
 The codes were written in Python and made use of the implemented boundary integral operators in the C++ library Bempp (see \cite{Bempp}).
\subsection{Spherically symmetric scattering: an example with an accurate reference solution}

For our first numerical example we choose $\Omega$ to be the exterior of the unit sphere. We consider a spherically symmetric incident wave $\uinc$ on the interval $[0,4]$ given by
\begin{align*}
\uinc(x,t)=\dfrac{e^{-5(\norm{x}-(3-t))^2}}{\norm{x}}.
\end{align*}
Since constant functions are eigenfunctions of the boundary operators, (see \cite{N01}), and also of the transfer operator of the generalized boundary condition in the cases (A)--(C), the scattered wave $u$ and the corresponding boundary densities $(\varphi,\psi)$ will then be constant on the unit sphere, i.e.~spherically symmetric.   \LB{Therefore, recalling that $\psi = \partial_t u$ on $\Gamma$, we have that the scattered field is given by
\begin{equation}
  \label{eq:exact_sphere}
  u(x,t) = \frac1{\norm{x}} u(y,t-(\norm{x}-1)) = \frac1{\norm{x}} \partial_t^{-1}\psi(y,t-(\norm{x}-1)), \qquad \text{for all } \norm{y} = 1.
\end{equation}}


\noindent To construct a reference solution, we now eliminate $\varphi$ in \eqref{Asoperator-t} and obtain
\begin{align}\label{L_form}
\left(L(\pt)+F(\pt)\right)\psi=-F(\pt) \pt\uinc +\pn \uinc,
\end{align}
where the corresponding operator $L(s)$ in the frequency domain is the scaled exterior Dirichlet-to-Neumann operator 
\begin{align}
L(s)=s^{-1}\left(W(s)-(\dfrac{1}{2}I-K^T(s))V(s)^{-1}(\dfrac{1}{2}I-K(s))\right) = - s^{-1} \mathrm{DtN}(s).
\end{align}
The eigenvalue of this operator that corresponds to constant functions is given by (see \cite{N01})
\begin{align}\label{EigL}
L(s)\wpsi=\left(1+\dfrac{1}{s}\right)\wpsi.
\end{align}
Now, we can easily discretize \eqref{L_form} in time to compute a reference density $\psi^\text{ref}$ of arbitrary precision (since no space discretization is needed).  Then, making use of \eqref{eq:exact_sphere} yields an approximation of arbitrary precision to the scattered wave in a given point $P\in \Omega$. In our experiments we used $N=2^{16}$ time steps to obtain a reference solution.


\newpage We present results for the second order absorbing boundary condition (B2), with $\varepsilon=10^{-2}$.
In view of the fact that the proof of Theorem \ref{thm:err-cq-bem-uscat} also implies a pointwise error bound, we use a fixed point (in our experiments $P=(2,0,0)$) and then take the maximum error in time, i.e.

\begin{align*}
\max_{0\le n \le N}
\abs{\uscatau_h\left(P,t_n\right)- \uscat^\text{ref}\left(P,t_n\right) }
\end{align*}
 as a function of the time stepsize $\tau$ and the spatial mesh width $h$. 
 
Employing our fully discrete scheme with varying mesh sizes $h_j=2^{-j}$ for $j=0,...,4$ and number of time steps $N_j=2^{j}$ for $j=4,...,11$ yields the convergence plots shown in Figures \ref{fig:time_conv} and~\ref{fig:space_conv}. 
 
 \begin{figure}[h!]
 \label{Fig1}
 \centering
 	\includegraphics[trim = 1mm 1mm 1mm 1mm, clip,width=0.8\textwidth,height=0.55\textwidth]{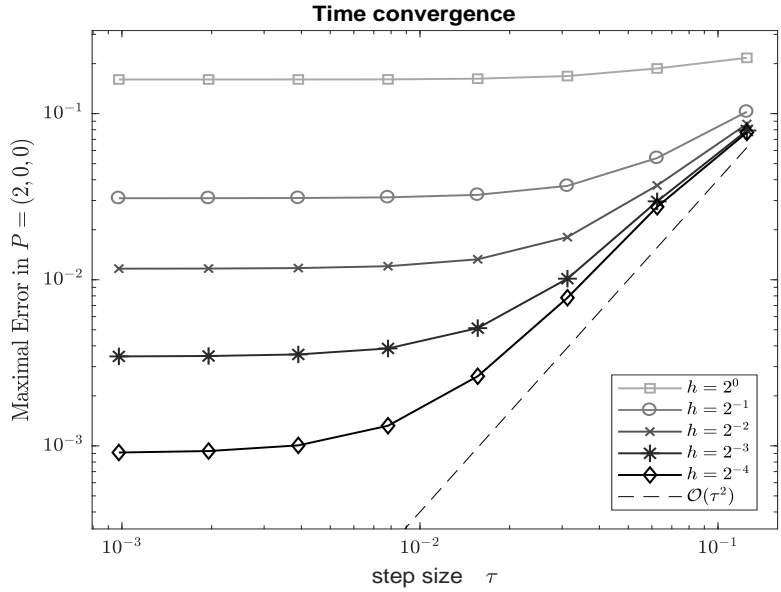}
 \caption{Time convergence plot of the fully discrete system, with varying degrees of freedom (dof) in space. }
 \label{fig:time_conv}
 \end{figure}

 \begin{figure}[h!]
 \label{Fig2}
 \centering
 	\includegraphics[trim = 1mm 1mm 1mm 1mm, clip,width=0.8\textwidth,height=0.55\textwidth]{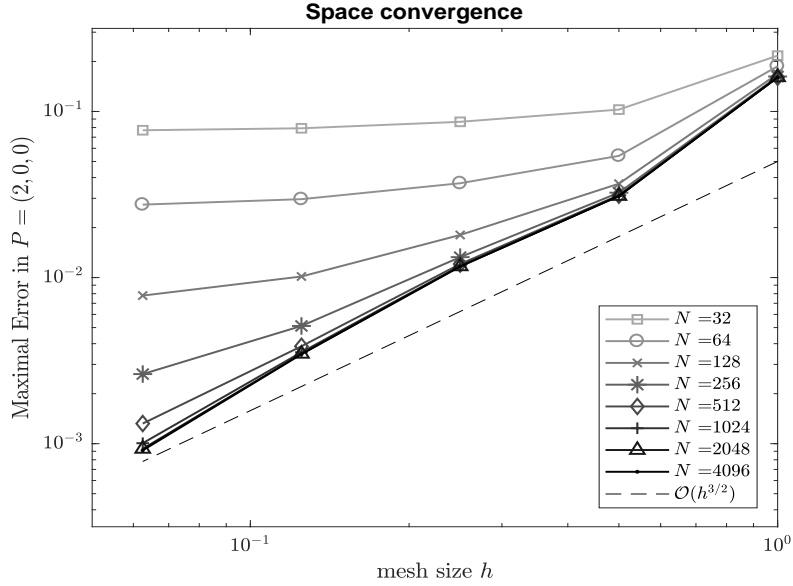}
 \caption{Space convergence plot of the fully discrete system, with varying number of time steps $N$.}
  \label{fig:space_conv}
 \end{figure}




\newpage
\subsection{Scattering of a plane wave from a halfpipe for different boundary conditions}
In the second example, we consider the scattering from a "halfpipe" shape, with a length of 1, a width of 0.5 and a height of 0.5 (as seen from above in Figure \ref{fig:wave_frames}). We discretize the surface with a grid consisting of about $10^3$ nodes. 
The incident wave is chosen as a plane wave, given by
\begin{align*}
\uinc(x,t):=e^{-100(x\cdot a-(t-t_0))^2}
\end{align*}
with $t_0=1$ and $a=(0,-1,0)$.

We employ the three generalized impedance boundary conditions (A),(B1),(C) and choose the corresponding parameters $\eps=10^{-1}$ in (A), $\eps=10^{-1}$ in (B1) and $m=\alpha=k=1$ in (C).
%
%

\begin{figure}[h!]
	\centering
\includegraphics[trim = 10mm 1mm 15mm 2mm, clip,width=1.0\textwidth,height=0.3\textwidth]{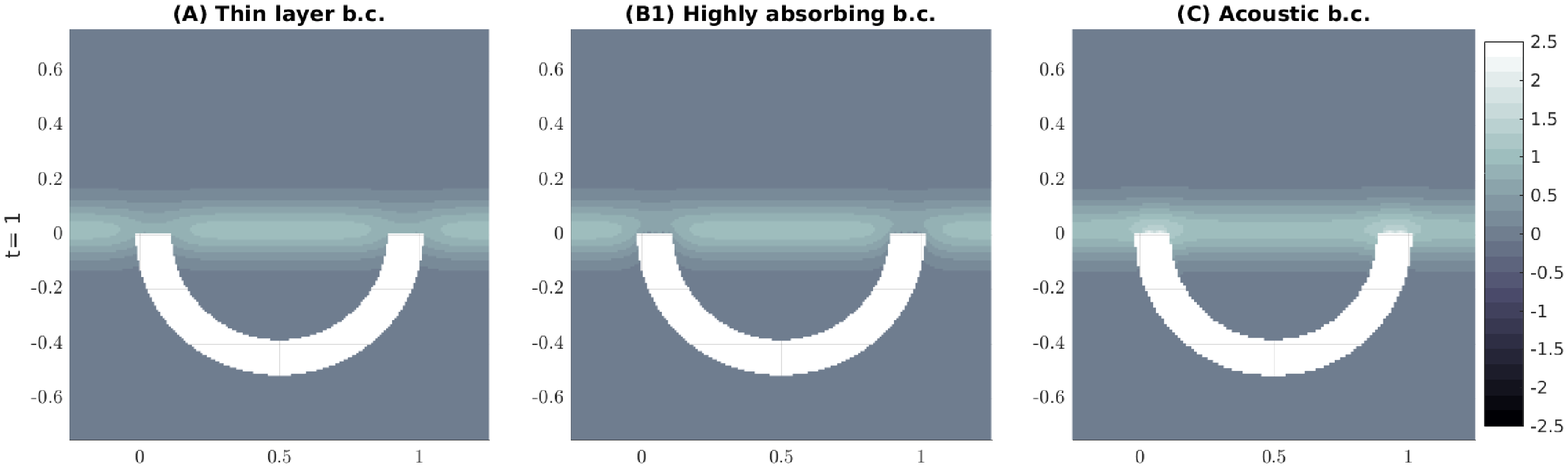}

\includegraphics[trim = 10mm 1mm 15mm 8mm, clip,width=1.0\textwidth]{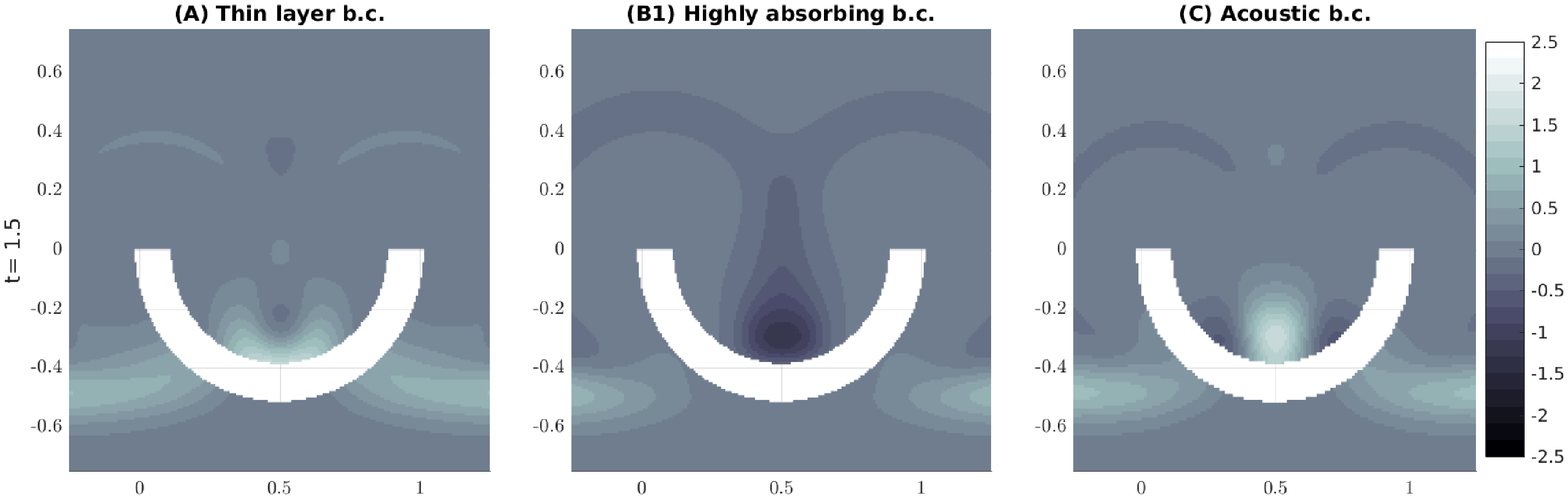}

\includegraphics[trim = 10mm 1mm 15mm 8mm, clip,width=1.0\textwidth]{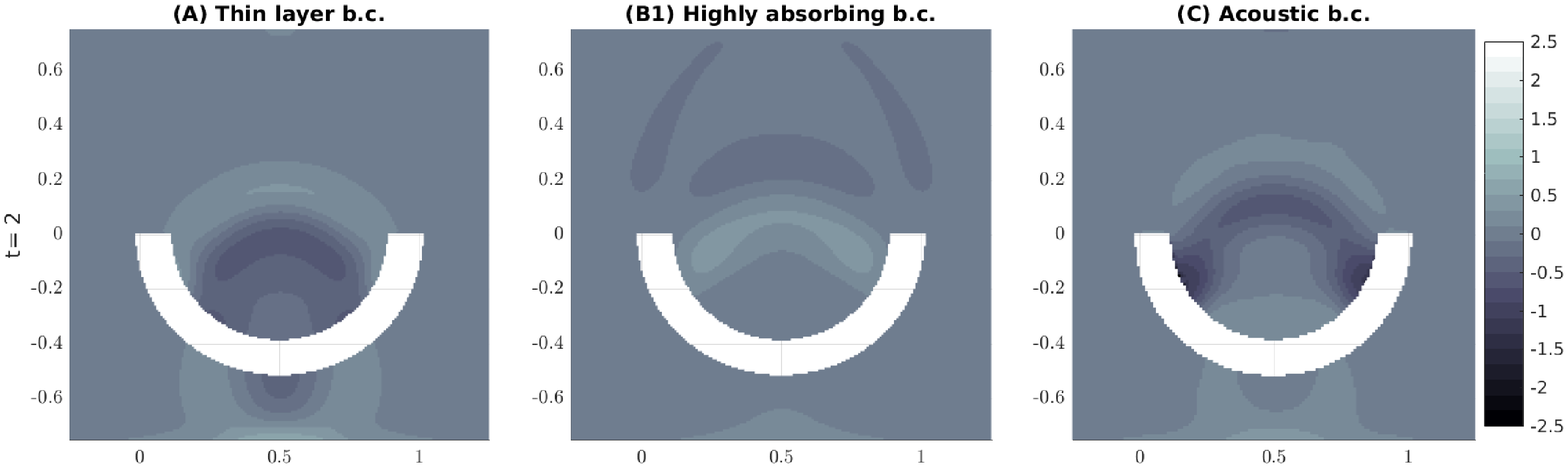}

\includegraphics[trim = 10mm 1mm 15mm 8mm, clip,width=1.0\textwidth]{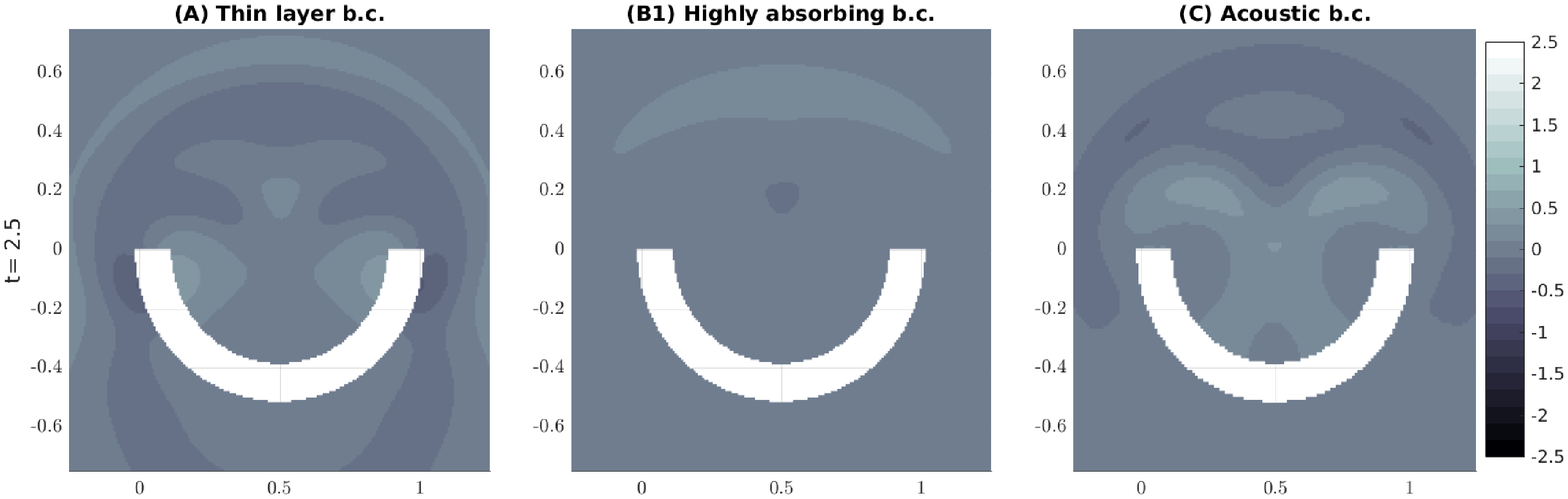}

\caption{Scattering of a plane wave from a "halfpipe" shape, with height $0.5$, where each column represents a different boundary condition. Shown is the plane $z=0.25$, the middle of the scatterer.}
\label{fig:wave_frames}
\end{figure}

\subsection{Acknowledgement}
We thank Bal\'azs Kov\'acs for helpful discussions. The second and third authors are funded by the Deutsche Forschungsgemeinschaft (DFG, German Research Foundation) -- Project-ID 258734477 -- SFB 1173.

\newpage 

  \bibliography{Lit}
 \bibliographystyle{authordate1}
  \addcontentsline{toc}{section}{Literature}

\normalsize
\end{document}